\documentclass{article}
\usepackage[utf8]{inputenc}
\usepackage{chicago}
\usepackage{amsthm}
\usepackage{ragged2e}
\usepackage{amsmath}
\usepackage{amsthm}
\usepackage{amssymb}
\usepackage{graphicx}
\usepackage{ragged2e} 

\usepackage{xcolor, colortbl}
\usepackage{array, multirow, multicol}

\newtheorem{remark}{Remark}%

\newtheorem{theorem}{Theorem}%
\newtheorem{definition}{Definition}%

\newtheorem{corollary}{Corollary}%

\newtheorem{proposition}{Proposition}%

\title{Topological indices in Random Spiro Chains}
\author{Saylé Sigarreta\thanks{Saylé Sigarreta Author. Email: sayle.sigarretar@alumno.buap.mx}, Saylí Sigarreta and Hugo Cruz-Suárez.}
\date{April 2022}

\begin{document}

\maketitle
\textbf{Abstract:} Let $G=(V(G),E(G))$ denote a graph, many important topological indices can be defined as \begin{center}
    $TI(G)= \sum_{v\in V(G)} h(d_{v})^{a}$,
\end{center} or \begin{center}
    $TI(G)= \sum_{vu\in E(G)} f(d_{v},d_{u})^{a}$.
\end{center} 
In this paper, we study these kinds of topological indices in random spiro chains via a martingale approach. In which their explicit analytical expressions of the exact distribution, expected value and variance are obtained. As $n$ goes to $\infty$, the asymptotic normality of topological indices of a random spiro chain is established through
the Martingale Central Limit Theorem. In particular, we compute the Nirmala, Sombor, Randić and Zagreb index for a random spiro chain along with their comparative analysis.

\section{Introduction}\label{sec1}

\noindent
A graph $G$ is determined by two sets $(V(G),E(G))$, the set of nodes ($V(G)$) and edges ($E(G)$). The edges and nodes are interpreted according to the problem to be modeled. In particular, a molecular graph is a simple graph such that its vertices correspond to the atoms and
the edges to the bonds of a molecule, where a simple graph is a graph without directed, weighted or multiple
edges, and without self-loops. Topological indices numerically quantify aspects of these graphs for multiple purposes, such as sparseness, regularity, and centrality. In addition, the first  and second Zagreb indices appeared for the first time in \cite{5}, then it is defined in \cite{6} the Randi\'c index. These indices are mostly historical and well-known indices which have been widely
used to predict the properties of compounds; since have been proved to have a wide range of functions as topological variables supported by chemical experiment data. In \cite{18}, a novel topological index was introduced via a geometric approach, named Sombor index defined as 

\hfill  
\medskip

\begin{center}
$S O(G)=\displaystyle\sum_{u v \in E(G)} \sqrt{\left(d_{u}\right)^{2}+\left(d_{v}\right)^{2}}$,
\end{center}

\hfill  
\medskip

\noindent
where $d_{v}$ is the degree of a vertex $v$. Nowadays, several graph invariants related to the Sombor index have been presented. For example, in \cite{19}, Kulli introduced the Nirmala index of a graph $G$ as follows
\medskip

\begin{equation}\label{nir}
     N(G)=\displaystyle \sum _{u v \in E(G)} \sqrt{d_{u}+d_{v}}.
\end{equation}

\medskip

\noindent
Recent work on the Nirmala index can be consulted in \cite{199}, \cite{20} and \cite{21}. In general, topological indices that can be constructed for static and random graphs represent a major part of the current research in mathematical chemistry and chemical graph theory.
\medskip

\noindent
On the other hand, the martingale theory is a very powerful and deep mathematical tool. The concept was
introduced by Paul Lévy in 1934, and was given its name in 1939 by J. André Ville. The
development of a whole theory around martingales is due to Joseph L. Doob. Nowadays, the concept of martingales is well-known. In particular,  there are martingale central limit theorems, which give conditions under which the whole process is approximately normally distributed. Actually, in \cite{22} and \cite{23} the authors used a martingale approach to study topological indices, such as the Zagreb, Gordon-Scantlebury and Platt indices.
\hfill
\medskip

\noindent
 In \cite{n1}, \cite{n2}, \cite{n3}, \cite{n4}, \cite{n6}, \cite{n5}, \cite{n7}, \cite{n8}  and \cite{n9}  the authors studied topological indices in random chains. In particular, spiro compounds are an important class of 
cycloalkanes in organic chemistry. The derivatives of spiros are fairly often seen chemicals, which may be utilized in organic synthesis, drug synthesis,
heat exchanger, etc. Motivated for the above information, we make researches on topological indices in random spiro 
chains. In this paper, our goal is to associate a martingale to the topological index, so that, the properties that can be deduced from the martingale are useful to show those of the topological index. We  first establish exact formulas for the expected value, variance and the exact distribution of topological indices in random spiro chains. Moreover, we find a general result for the asymptotic distribution via this approach. Finally, as applications, using the Nirmala, Sombor, Randić and Zagreb index, the results are given for random spiro chains (see in Definition \ref{d1}). 

\begin{definition}\label{d1}
The random spiro chain $RSC_{n}=RSC\left(n, p_{1}, p_{2}, p_{3}\right)$ with $n$ hexagons is constructured by the following way:
\begin{itemize}
    \item $RSC_{1}$ is a hexagon and $RSC_{2}$ contains two hexagons, see Figure \ref{f7}.
    \item For every $n>2, RSC_{n}$ is constructured by attaching one hexagon to $RSC_{n-1}$ in three ways, resulted in $RSC_{n}^{1}, R S C_{n}^{2}, R S C_{n}^{3}$ with probability $p_{1}, p_{2}$ and $p_{3}$ respectively, where $0  < p_{i}  < 1$ and $p_{1}+p_{2}+p_{3}=1$, see Figure \ref{f8}.
\end{itemize}

\end{definition} 

\begin{figure}[h!]

   \centering
    \includegraphics[width=0.5\textwidth]{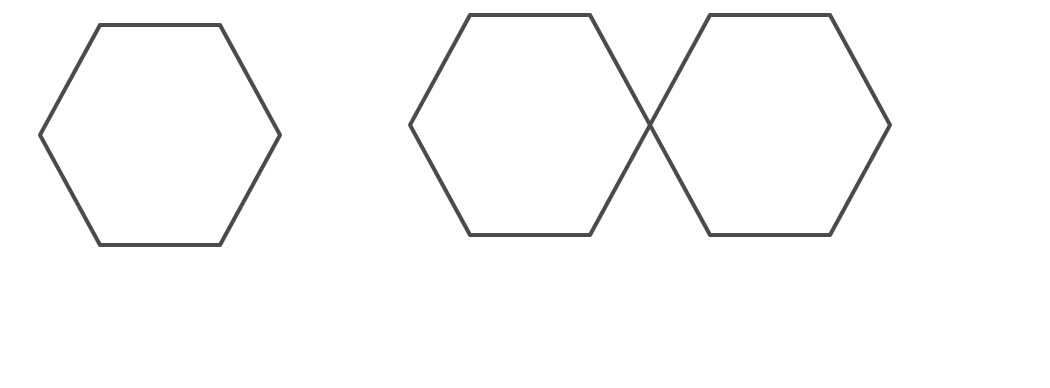}
    \caption{The graphs of $RSC_{1}$ and $RSC_{2}$.}

    \label{f7}
\end{figure}

\begin{figure}[h!]
    \centering
    \includegraphics[width=0.9\textwidth]{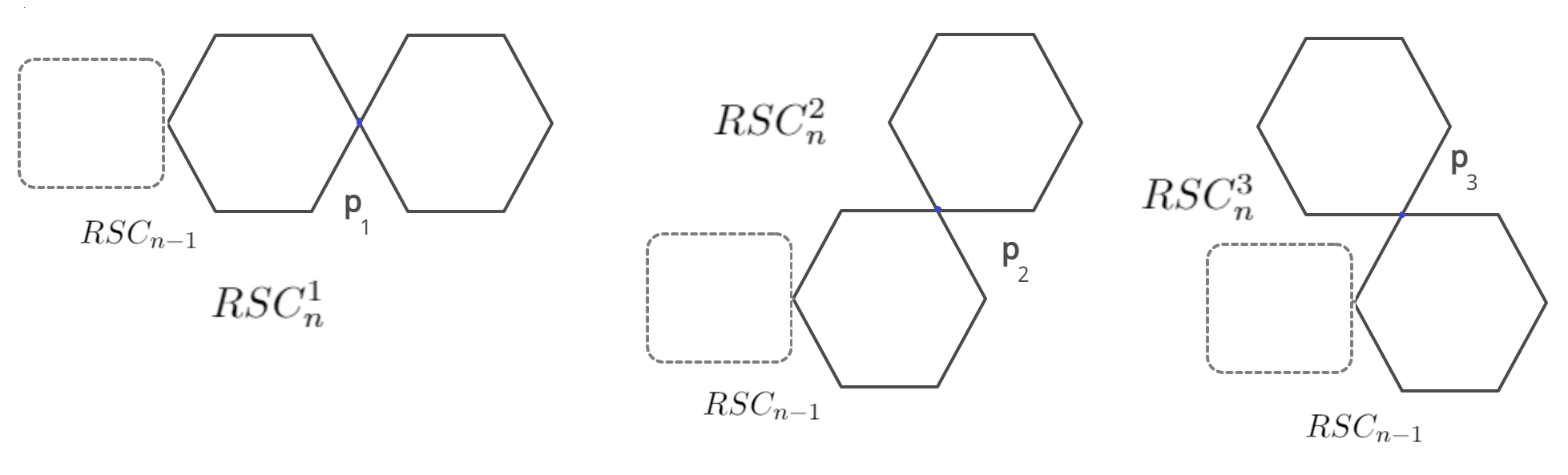}
   \caption{The three link ways for $RSC_{n} (n > 2)$.}
    \label{f8}
\end{figure}
\section{Topological indices in random spiro chains.}\label{sect}
\noindent
Let $G=(V(G),E(G))$, many important topological indices can be defined as
\begin{equation}\label{sa1}
   TI(G)= \displaystyle\sum_{v\in V(G)} h(d_{v})^{a},
\end{equation}
or
\begin{equation}\label{sa2}
   TI(G)= \displaystyle\sum_{vu\in E(G)} f(d_{v},d_{u})^{a},
\end{equation}
    
\noindent
where $a\in \mathbb{R} $, $h:\{1,2, \dots \} \rightarrow (0, \infty)$ and $f:\{1,2, \dots \} \times \{1,2, \dots \} \rightarrow (0, \infty)$ is any symmetric function. The main topological indices of the form (\ref{sa1}) and (\ref{sa2}) are:
\begin{itemize}
    \item If $ h(t)=t$ and $a=2$ then $TI(G)$ is the first Zagreb index.
    \item If $ h(t)=t$ and $a=-1$ then $TI(G)$ is the inverse degree index.
    \item If $ h(t)=t$ and $a=3$ then $TI(G)$ is  the forgotten index.
     \item If $ h(t)=t$ and $a \in \mathbb{R}$ then $TI(G)$ is the variable first Zagreb index.
     \item If $ f(x,y)=xy$ and $a=1$ then $TI(G)$ is the second Zagreb  index.
     \item If $f(x,y)=xy$ and $a=-1/2$ then $TI(G)$ is the usual Randić index.
     \item If $f(x,y)=x+y$ and $a=-1/2$ then $TI(G)$ is the sum-connectivity index.
     \item If $f(x,y)=x+y$ and $a=-1$ then $2TI(G)$ is the harmonic index.
     \item If $f(x,y)=x+y$ and $a\in\mathbb{R}$ then $TI(G)$ is the variable sum-connectivity index.

\end{itemize}

\begin{remark}
Note that the Nirmala index (\ref{nir}) is the reverse version of the sum-connectivity index. In addition, the Nirmala index is the variable sum-connectivity index, for $a = 1 / 2$.
\end{remark}

\begin{theorem}\label{t1}
Let $RSC_{n}=RSC\left(n, p_{1},  p_{2}, p_{3}\right)$ with $n \geq 2$ be a random spiro chain. Then \hfill
\medskip
\begin{equation*}
    \mathbb{E}(TI_{n})=TI_{2}+\alpha(n-2),
\end{equation*}
\begin{equation*}
V(TI_{n})=(\beta-\alpha^{2})(n-2),
\end{equation*}
\medskip

\noindent
where  $i=1,2,3$, $TI_{n}=TI(RSC_{n})$, $TI_{n,i}=TI(RSC_{n}^{i})$, $\alpha_{i}=TI_{3,i}-TI_{2}$, $\alpha=\displaystyle\sum _{i=1}^{3}\alpha_{i}p_{i}$ and $\beta=\displaystyle\sum _{i=1}^{3} \alpha_{i}^{2}p_{i}$.
\end{theorem}

\begin{proof}
Let $n \geq 3$ and $L_{n}$ denote a random variable with range
$\{1,2,3\}$ and let $p_{i}= \mathbb{P}(L_{n} = i)$ and $L_{2}$ denote the initial link, i.e., $L_{n}$ denote  the  link  selected  at  time $n$. Note that, at time $n-1$ we have

\hfill
\medskip

\begin{center}
    $\underbrace{\underbrace{H L_{2} H}_{{RSC_{2}}} L_{3}  H L_{4} H L_{5} \dots L_{n-1}H}_{{RSC_{n-1}}}.$
\end{center}
\hfill
\medskip

\noindent
Then, at time $n$, we obtain 
\hfill
\medskip

\begin{center}
    $\underbrace{\underbrace{\underbrace{H L_{2} H}_{{RSC_{2}}} L_{3}  H L_{4} H L_{5} \dots L_{n-1}H}_{{RSC_{n-1}}} L_{n} H}_{{RSC_{n}}}.$
\end{center}
\hfill
\medskip

\noindent
 Therefore, we must pay attention to the change in the calculation of the topological index by joining $H$ with $H$ via $L_{n}$. Let $n \geq 3$ and $i=1,2,3$, then based on this approach, by the definition of a random spiro chain and $TI(G)$ in Equation (\ref{sa1}) and (\ref{sa2}), we obtain the following almost-sure recursive relation between $TI_{n-1}$ and $TI_{n}$, conditional on the event that the at time $n$ the link $i$ is selected and $\mathbb{F}_{n-1}$
 
 \hfill
\medskip
\begin{equation*}
     TI_{n,i}-TI_{n-1}=TI_{3,i}-TI_{2},
 \end{equation*}

\hfill
\medskip

\noindent
where $\mathbb{F}_{n-1}$ denotes the $\sigma$-field generated by the history of the growth of the random spiro chain in the first $n-1$ stages. Now, we take the expectation with respect to $L_{n}$ to get,

\begin{eqnarray*}
  \mathbb{E}(TI_{n} \mid \mathbb{F}_{n-1})&=&\displaystyle \sum _{i=1}^{3}(TI_{n-1}+\alpha_{i})p_{i}\\
    &=&TI_{n-1}+\sum _{i=1}^{3}\alpha_{i}p_{i},\\
    \end{eqnarray*}

\noindent
where, $\alpha_{i}=TI_{3,i}-TI_{2}$. Then, taking expectation, we obtain a recurrence relationship for $\mathbb{E}(TI_{n})$,

\begin{equation}\label{e2}
    \mathbb{E}(TI_{n})=\mathbb{E}(TI_{n-1})+\sum _{i=1}^{3}\alpha_{i}p_{i}.
\end{equation}

\medskip

\noindent
We solve Equation $(\ref{e2})$ with the initial value $\mathbb{E}(TI_{2})=TI_{2}$ and we obtain the result stated in the theorem,
\hfill
\medskip
\begin{equation*}
    \mathbb{E}(TI_{n})=TI_{2}+\alpha(n-2),
\end{equation*}

\medskip

\noindent
where $\alpha=\displaystyle\sum _{i=1}^{3}\alpha_{i}p_{i}$. The expressions for $\mathbb{E}(TI^{2}_{n})$ follow in a similar manner,
\medskip
\newline

\begin{eqnarray*}
 \mathbb{E}(TI^{2}_{n} \mid \mathbb{F}_{n-1})&=&\displaystyle \sum _{i=1}^{3} (TI_{n-1}+\alpha_{i})^{2}p_{i}\\
  &=&\displaystyle \sum _{i=1}^{3} TI^{2}_{n-1}p_{i}+2TI_{n-1}\alpha_{i}p_{i}+\alpha_{i}^{2}p_{i}\\
   &=&TI^{2}_{n-1}+2TI_{n-1}\alpha+\beta,\\
\end{eqnarray*}

\noindent
where $\beta=\displaystyle\sum _{i=1}^{3} \alpha_{i}^{2}p_{i}$, thus
\medskip
\newline

\begin{eqnarray*}
\mathbb{E}(TI^{2}_{n})&=&\mathbb{E}(TI^{2}_{n-1})+2\alpha\mathbb{E}(TI_{n-1})+\beta\\
  &=&\mathbb{E}(TI^{2}_{n-1})+2\alpha TI_{2}+2\alpha^{2}(n-3)+\beta,
\end{eqnarray*}

\noindent
with $\mathbb{E}(TI^{2}_{2})=TI^{2}_{2}$, then iterating, it is obtained that 

\medskip
\begin{center}
    $\mathbb{E}(TI^{2}_{n})=TI^{2}_{2}+(2\alpha TI_{2}+\beta)(n-2)+(n-3)(n-2)\alpha^{2}.$
\end{center}
\medskip

\noindent
The variance of $TI_{n}$ is obtained immediately by taking the difference between $\mathbb{E}(TI^{2}_{n})$ and $\mathbb{E}(TI_{n})^{2}$,

\medskip

\begin{eqnarray*}
V(TI_{n})&=&\beta(n-2)+ \left((n-2)(n-3)-(n-2)^{2}\right)\alpha^{2}\\
       &=&(\beta-\alpha^{2})(n-2),\\
\end{eqnarray*}

\noindent
proving the theorem.
\end{proof}

\hfill
\medskip

\noindent
Note that $\beta - \alpha^
{2}=0$ if and only if $\alpha_{1}=\alpha_{2}=\alpha_{3}$ if and onl if $TI_{n}=TI_{2}+\alpha(n-2)$ a.s. with $n\geq 2$ (a deterministic sequence). Now, we exploit a martingale formulation to investigate the asymptotic behavior 
of $TI_{n}$ when $\beta - \alpha^
{2}>0$. The key idea is to consider a transformation $M_{n}$ and we require that the transformed
random variables form a martingale in the next proposition.

\begin{proposition}\label{l1}
For $n \geq 2$, $\{M_{n}=TI_{n}-\alpha(n-2)\}_{n}$ is a martingale with respect to $\mathbb{F}_{n}$.
\end{proposition}

\begin{proof} Firstly, observe that $\mathbb{E}(|M_{n}|)< +\infty$. Then, by Theorem \ref{t1},
$$
\begin{aligned}
\mathbb{E}\left(TI_{n}-\alpha(n-2) \mid \mathbb{F}_{n-1}\right) &= \mathbb{E}\left(TI_{n} \mid \mathbb{F}_{n-1}\right)-\alpha(n-2) \\
&= TI_{n-1}+ \alpha-\alpha(n-2) \\
&=TI_{n-1}-\alpha(n-3).
\end{aligned}
$$
The proof is completed.

\end{proof}
\hfill
\medskip

\noindent
We use the notation $\stackrel{D}{\longrightarrow}$ to denote convergence in distribution and $\stackrel{P}{\longrightarrow}$ to denote convergence in probability. The random variable $\mathrm{N}\left(\mu, \sigma^{2}\right)$ appears in the following theorem for the normal distributed with mean $\mu$ and variance $\sigma^{2}$.

\begin{theorem}\label{t2}
As $n \rightarrow \infty$,
\hfill
\medskip

\begin{center}
    $\frac{TI_{n}-(n-2) \alpha}{\sqrt{n}} \stackrel{D}{\longrightarrow}N(0,\beta - \alpha^
{2})$.
\end{center}
\end{theorem}
\begin{proof}
 Note that, for $j \geq 3$ and $i=1,2,3$, we have
\begin{center}
    $|\nabla M_{j}|=|\nabla TI_{j}-\alpha|\leq 2~ \underset{i}{max} \{|\alpha_{i}|\}$,
\end{center}
\hfill
\medskip

\noindent
where $\nabla M_{j}=M_{j}-M_{j-1}$ and $\nabla TI_{j}=TI_{j}-TI_{j-1}$. Then, as  $n$ goes to $\infty$
\begin{center}
    $\displaystyle\lim _{n \rightarrow \infty} \frac{\left|\nabla M_{j}\right|}{\sqrt{n}}=0$.
\end{center}

\noindent
That is, given $\varepsilon>0$, there exists an $n_{0}(\varepsilon)>0$ such that, the sets $\left\{ |\nabla M_{j}|  > \varepsilon\sqrt{n} \right\}$ are empty for all $n>n_{0}(\varepsilon)$. In what follows, we conclude that 
\begin{center}
    $U_{n}:=\frac{1}{n}\displaystyle\sum_{j=3}^{n} \mathbb{E}\left(\left(\nabla M_{j}\right)^{2} \mathbb{I}_{\left\{ |\nabla M_{j}|  > \varepsilon\sqrt{n} \right\}} \mid \mathbb{F}_{j-1}\right),$
\end{center} 
converges to 0 almost surely, hence, $U_{n} \stackrel{P}{\longrightarrow}0$. As a result, the Lindeberg’s condition is verified. Next, the conditional variance condition is given by

$$
V_{n}:=\frac{1}{n}\sum_{j=3}^{n} \mathbb{E}\left(\left(\nabla M_{j}\right)^{2} \mid \mathbb{F}_{j-1}\right) \stackrel{P}{\longrightarrow} \beta - \alpha^
{2}.
$$
\noindent
Note that,

\begin{center}
    $$
\begin{aligned}
\frac{1}{n}\sum_{j=3}^{n} \mathbb{E}\left(\left(\nabla M_{j}\right)^{2} \mid \mathbb{F}_{j-1}\right)=& \frac{1}{n} \sum_{j=3}^{n} \mathbb{E}\left((\nabla TI_{j}-\alpha)^{2}\mid \mathbb{F}_{j-1}\right) \\
=& \frac{1}{n} \sum_{j=3}^{n}\sum_{i=1}^{3}(\alpha_{i}-\alpha)^{2}p_{i}\\
=& \frac{n-2}{n}\sum_{i=1}^{3}(\alpha_{i}-\alpha)^{2}p_{i}.
\end{aligned}
$$
\end{center}

\noindent
By the Martingale Central Limit Theorem \cite{26}, we thus obtain the stated result, since $$
\begin{aligned}
\sum_{i=1}^{3}(\alpha_{i}-\alpha)^{2}p_{i}=&\sum_{i=1}^{3} \alpha_{i}^{2} p_{i}-2\alpha \sum_{i=1}^{3} \alpha_{i} p_{i}+\alpha^{2}\\
=& \beta-\alpha^{2}.
\end{aligned}
$$

\end{proof}

\noindent
Then we may use Theorem \ref{t2} to find the following result.

\begin{corollary}\label{c0}
As $n \rightarrow \infty$,
\hfill
\medskip

\begin{center}
    $\frac{TI_{n}-\mathbb{E}(TI_{n})}{\sqrt{V(TI_{n})}} \stackrel{D}{\longrightarrow}N(0,1)$.
\end{center}

\end{corollary}

\noindent
The following theorem gives further details on the distribution for topological indices in random spiro chains. Here, $M_{R}(\cdot)$ denotes the moment generating function of a random variable $R$.

\hfill
\medskip

\noindent

\begin{theorem}\label{t3}
Let $RSC_{n}$ with $n \geq 2$ be a random spiro chain. Then,
\hfill
\medskip
\begin{center}
$TI_{n}=TI_{2}+a^{T}X$,
\end{center}
\hfill
\medskip

\noindent
where $a^{T}=(\alpha_{1},\alpha_{2},\alpha_{3})$ and $X=(X_{1},X_{2},X_{3})$ is a multinomial random variable  with parameters $n-2$ and $(p_{1},p_{2},p_{3})$.
\end{theorem}
\begin{proof} Let $t \in \mathbb{R}$, note that,
 $$
\begin{aligned}
\mathbb{E}\left(e^{tTI_{n}} \mid \mathbb{F}_{n-1}\right)=&\sum_{i=1}^{3} e^{tTI_{n-1}}  e^{t \alpha_{i}} p_{i}\\
= &~e^{tTI_{n-1}}\sum_{i=1}^{3}  e^{t \alpha_{i}} p_{i}.\\
\end{aligned}
$$
\noindent
Thus, we can conclude that

\begin{center}
    $M_{TI_{n}}(t)=M_{TI_{n-1}}(t)\displaystyle\sum_{i=1}^{3}  e^{t \alpha_{i}} p_{i}.$
\end{center}
\hfill
\medskip

\noindent
 We may therefore write, 
\hfill
\medskip

\noindent
$$
\begin{aligned}
M_{TI_{n}}(t)=& M_{TI_{2}}(t)\left(\displaystyle\sum_{i=1}^{3}  e^{t \alpha_{i}} p_{i}\right)^{n-2} \\
=&M_{TI_{2}}(t)M_{X}(\alpha_{1}t,\alpha_{2}t, \alpha_{3}t)\\
=&M_{TI_{2}}(t)M_{a^{T}X}(t),
\end{aligned}
$$
\noindent
which completes the proof.
\end{proof}

\noindent
It is useful to note that the approximation given in Corollary \ref{c0} is identical to the one
obtained by the following method. Let $n\geq 2$, by Theorem \ref{t3} we have that
\hfill
\medskip
\begin{center}
$TI_{n}=TI_{2}+a^{T}X$.
\end{center}
\hfill
\medskip

\noindent
 It follows from the Central Limit Theorem to the case of random vectors \cite{27} that $X$ is asymptotically distributed according to a multivariate normal distribution with mean $\mathbb{E}(X)$
and covariance matrix $V(X)$. Consequently, $a^{T}X$ is asymptotically distributed according to a normal distribution with mean $a^{T}\mathbb{E}(X)$
and variance $a^{T}V(X)a$. Then, as $n$ goes to $\infty$
\hfill
\medskip
\begin{center}
$\frac{TI_{n}-TI_{2}-a^{T}\mathbb{E}(X)}{\sqrt{a^{T}V(X)a}}=\frac{TI_{n}-\mathbb{E}(TI_{n})}{\sqrt{V(TI_{n})}}\xrightarrow{D} N(0,1)$.
\end{center}
 \hfill
\medskip

\noindent
\section{ Interpretation of the results and examples.}
\noindent
The conclusion of Section \ref{sect} can be stated as follows.
\begin{theorem}\label{galle}
Let $R S C_{n}$ with $n \geq 2$ be a random spiro chain and $a \in \mathbb{R}$. Then
\begin{center}
   $T I(G)=\displaystyle\sum_{v \in V(G)} h\left(d_{v}\right)^{a}=2h(2)^{a}-h(4)^{a}+(4h(2)^{a}+h(4)^{a})n$,
\end{center}

\noindent
and

\begin{center}
$T I(G)=\displaystyle\sum_{v u \in E(G)} f\left(d_{v}, d_{u}\right)^{a}=A+BX+Cn$,
 \end{center}

\begin{center}
$\mathbb{E}(TI_{n})=A+(Bp_{1}+C)n-2Bp_{1}$,
 \end{center} 
 \begin{center}
$V(TI_{n})=B^{2}p_{1}(1-p_{1})(n-2)$,
 \end{center}

\noindent
where $A=4f\left(2, 2\right)^{a}-4f\left(2, 4\right)^{a}$, $B=f\left(2, 2\right)^{a}-2f\left(2, 4\right)^{a}+f\left(4, 4\right)^{a}$, $C=2f\left(2, 2\right)^{a}+4f\left(2, 4\right)^{a}$ and $X$ has a binomial distribution with parameters $n-2$ and $p_{1}$.
\end{theorem} 

\begin{proof}
 As can be seen from the results obtained in Section \ref{sect}, we need to find $TI_{2}$ and $\alpha_{i}$ with $i=1,2,3$. By the definition of $TI_{n}$ in Equation (\ref{sa1}), $RSC_{2}$ and $RSC_{3}^{i}$, we have that,

\begin{center}
$TI_{2}=10h(2)^{a}+h(4)^{a}$,
\end{center}

\begin{center}
$T_{3,1}=T_{3,2}=T_{3,3}=14h(2)^{a}+2h(4)^{a}$.
\end{center}
\noindent
Then,
\begin{center}
$\alpha_{1}=\alpha_{2}=\alpha_{3}=4h(2)^{a}+h(4)^{a}$.
\end{center}
\noindent
On the other hand, by the definition of $TI_{n}$ in Equation (\ref{sa2}), it follows that,

\begin{center}
$TI_{2}=8f(2,2)^{a}+4f(2,4)^{a}$,
\end{center}

\begin{center}
$T_{3,1}=T_{3,2}=10f(2,2)^{a}+8f(2,4)^{a}$,
\end{center}

\begin{center}
$T_{3,3}=11f(2,2)^{a}+6f(2,4)^{a}+f(4,4)^{a}$.
\end{center}

\noindent
Then,
\begin{center}
$\alpha_{1}=\alpha_{2}=2f(2,2)^{a}+4f(2,4)^{a}$,
\end{center}
\begin{center}
$\alpha_{3}=3f(2,2)^{a}+2f(2,4)^{a}+f(4,4)^{a}$.
\end{center}
\noindent
In each case, applying Theorem \ref{t1}, \ref{t2} and \ref{t3}, we verify the results.
\end{proof}
\begin{remark}
Note that if $TI(G)= \displaystyle\sum_{v\in \in V(G)} h(d_{v})^{a}$ then $TI_{n}$ is a deterministic sequence and if $TI(G)= \displaystyle\sum_{vu\in E(G)} f(d_{v},d_{u})^{a}$, we have that $TI_{n}=TI_{2}+\alpha_{1}(n-2)$  with $n\geq 2$ (a deterministic sequence) if and only if $f(2,2)^{a}+f(4,4)^{a}=2f(2,4)^{a}$. In particular, if $a=1$ then taking $f(x,y)=x^{\theta}+y^{\theta}$ with $\theta \in \mathbb{R}$, it is verified that $f(2,2)+f(4,4)=2f(2,4)$; which make sense, since $TI(G)= \displaystyle\sum_{vu\in E(G)} d_{u}^{\theta}+d_{v}^{\theta}=\displaystyle\sum_{v\in E(G)} d_{v}^{\theta+1} $
\end{remark}

\noindent
Now, in order to apply Theorem \ref{galle}, we present the following corollaries.

\begin{corollary}\label{c2}
Let $RSC_{n}=RSC\left(n, p_{1}, p_{2}, p_{3}\right)$ be a random spiro chain and $N_{n}$ be the Nirmala index of a $RSC_{n}$, with $n \geq 2$. Then

\begin{equation*}
    N_{n}=8-4 \sqrt{6}+(2-2 \sqrt{6}+2\sqrt{2}) X+(4+4 \sqrt{6})n,
\end{equation*}

\begin{equation*}
    \mathbb{E}\left(N_{n}\right)=8-4 \sqrt{6}+\left((2-2 \sqrt{6}+2\sqrt{2}) p_{1}+4+4 \sqrt{6}\right)n-2(2-2 \sqrt{6}+2\sqrt{2})p_{1},
\end{equation*}

\begin{equation*}
     V\left(N_{n}\right)=(2-2 \sqrt{6}+2\sqrt{2})^{2} p_{1}\left(1-p_{1}\right)(n-2),
\end{equation*}

\begin{equation*}
    \frac{N_{n}-\mathbb{E}\left(N_{n}\right)}{\sqrt{V\left(N_{n}\right)}} \xrightarrow{D} N(0,1),
\end{equation*}
\hfill
\medskip

\noindent
where $X$ has a binomial distribution with parameters $n-2$ and $p_{1}$.

\end{corollary}

\begin{corollary}\label{c6}
Let $RSC_{n}=RSC\left(n, p_{1}, p_{2}, p_{3}\right)$ be a random spiro chain and $M1_{n}$ be the first Zagreb index of a $RSC_{n}$, with $n \geq 2$. Then

\begin{equation}\label{zagr}
    M1_{n}=32n-8.
\end{equation}

\end{corollary}

\begin{corollary}\label{c3}
Let $RSC_{n}=RSC\left(n, p_{1}, p_{2}, p_{3}\right)$ be a random spiro chain and $R_{n}$ be the Randić index of a $RSC_{n}$, with $n \geq 2$. Then

\begin{equation*}
    R_{n}=2-\sqrt{2}+(3/4-\sqrt{2}/2) X+(1+\sqrt{2})n,
\end{equation*}

\begin{equation}\label{rand}
    \mathbb{E}\left(R_{n}\right)=2-\sqrt{2}+\left((3/4-\sqrt{2}/2) p_{1}+1+\sqrt{2}\right)n+(\sqrt{2}-3/2)p_{1},
\end{equation}

\begin{equation*}
     V\left(R_{n}\right)=(3/4-\sqrt{2}/2)^{2} p_{1}\left(1-p_{1}\right)(n-2),
\end{equation*}

\begin{equation*}
    \frac{R_{n}-\mathbb{E}\left(R_{n}\right)}{\sqrt{V\left(R_{n}\right)}} \xrightarrow{D} N(0,1),
\end{equation*}
\hfill
\medskip

\noindent
where $X$ has a binomial distribution with parameters $n-2$ and $p_{1}$.

\end{corollary}

\begin{corollary}\label{c4}
Let $RSC_{n}=RSC\left(n, p_{1}, p_{2}, p_{3}\right)$ be a random spiro chain and $S_{n}$ be the Sombor index of a $RSC_{n}$, with $n \geq 2$. Then

\begin{equation*}
    S_{n}=8\sqrt{2}-8\sqrt{5}+(6\sqrt{2}-4\sqrt{5}) X+(4\sqrt{2}+8\sqrt{5})n,
\end{equation*}

\begin{equation*}
    \mathbb{E}\left(S_{n}\right)=8\sqrt{2}-8\sqrt{5}+\left((6\sqrt{2}-4\sqrt{5}) p_{1}+4\sqrt{2}+8\sqrt{5}\right)n-2(6\sqrt{2}-4\sqrt{5})p_{1},
\end{equation*}

\begin{equation*}
     V\left(S_{n}\right)=(6\sqrt{2}-4\sqrt{5})^{2} p_{1}\left(1-p_{1}\right)(n-2),
\end{equation*}

\begin{equation*}
    \frac{S_{n}-\mathbb{E}\left(S_{n}\right)}{\sqrt{V\left(S_{n}\right)}} \xrightarrow{D} N(0,1),
\end{equation*}
\hfill
\medskip

\noindent
where $X$ has a binomial distribution with parameters $n-2$ and $p_{1}$.

\end{corollary}

\begin{corollary}\label{c5}
Let $RSC_{n}=RSC\left(n, p_{1}, p_{2}, p_{3}\right)$ be a random spiro chain and $M2_{n}$ be the second Zagreb index of a $RSC_{n}$, with $n \geq 2$. Then

\begin{equation*}
    M2_{n}=4X+40n-16,
\end{equation*}

\begin{equation}\label{zag}
    \mathbb{E}\left(M2_{n}\right)=\left(4p_{1}+40\right)n-8p_{1}-16,
\end{equation}

\begin{equation*}
     V\left(M2_{n}\right)= 16p_{1}\left(1-p_{1}\right)(n-2),
\end{equation*}

\begin{equation*}
    \frac{M2_{n}-\mathbb{E}\left(M2_{n}\right)}{\sqrt{V\left(M2_{n}\right)}} \xrightarrow{D} N(0,1),
\end{equation*}
\hfill
\medskip

\noindent
where $X$ has a binomial distribution with parameters $n-2$ and $p_{1}$.

\end{corollary}

\begin{remark}
In fact, we can see that (\ref{zagr}) and (\ref{rand}) are obtained in
\cite{n6}. Also, we can see that  Corollary \ref{c4} and (\ref{zag}) are obtained in \cite{n9} and \cite{n7}, respectively.
\end{remark}

\begin{remark}
For $n\geq 2$ and $p_{1}\in(0,1)$, it follows from Corollary \ref{c2}, \ref{c6}, \ref{c3}, \ref{c4} and \ref{c5} that

\hfill

\begin{center}
    $\mathbb{E}(R_{n}) \leq \mathbb{E}(N_{n})\leq \mathbb{E}(S_{n}) \leq \mathbb{E}(M1_{n})\leq \mathbb{E}(M2_{n})$ (see Figure \ref{fl}),
    
\end{center}

\medskip

\begin{center}
    $V(R_{n}) \leq V(N_{n})\leq V(S_{n}) \leq  V(M2_{n})$.
\end{center}
\end{remark}

\begin{figure}[h!]

   \centering
    \includegraphics[width=0.5\textwidth]{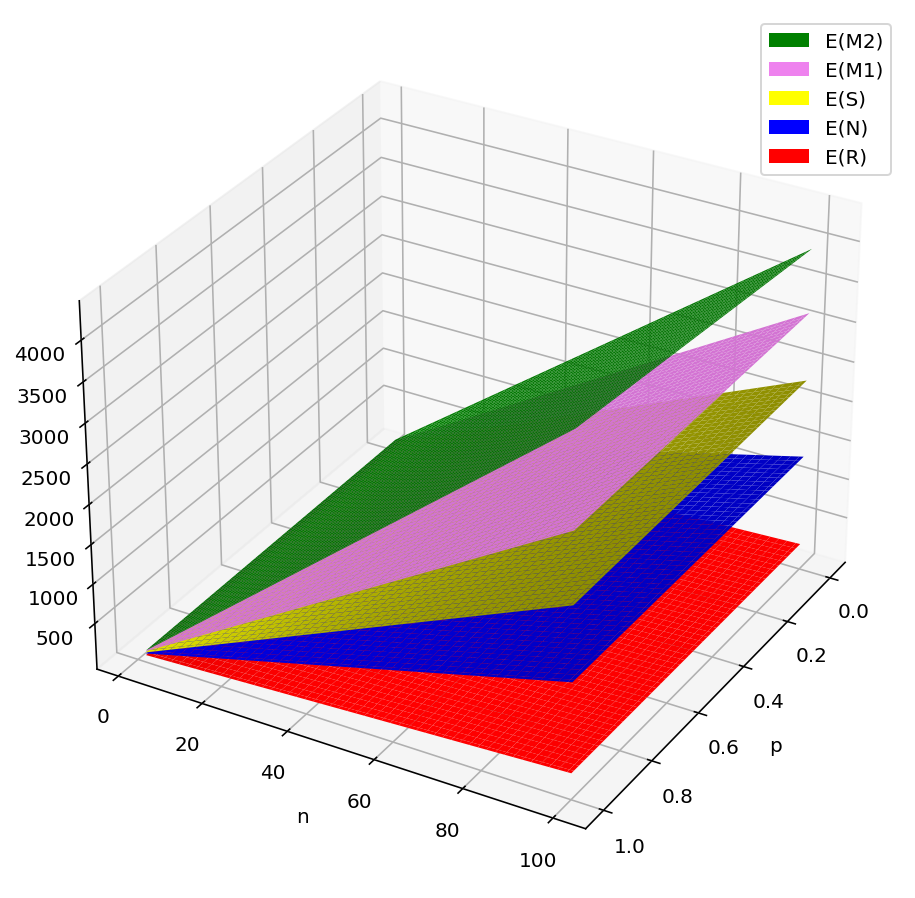}
    \caption{ Difference between $\mathbb{E}(R_{n}), \mathbb{E}(N_{n}), \mathbb{E}(S_{n}), \mathbb{E}(M1_{n})$ and $\mathbb{E}(M2_{n})$.} 
    \label{fl}
\end{figure}

\noindent
Finally, we conduct a numerical experiment to support the asymptotic behaviors developed in Corollaries \ref{c2}, \ref{c3}, \ref{c4} and \ref{c5}. Given a fixed $p_{1} \in (0,1)$, in each case, we independently generate  $5,000$ replications of a random spiro
chain after $n = 10,000$ evolutionary steps. For each simulated random spiro chain, its topological index is computed. The histogram of the sample data with a normal approximation curve are given in Figure \ref{f9},~\ref{f10},~\ref{f11} and \ref{f12}.

\begin{figure}[h!]

   \centering
    \includegraphics[width=0.5\textwidth]{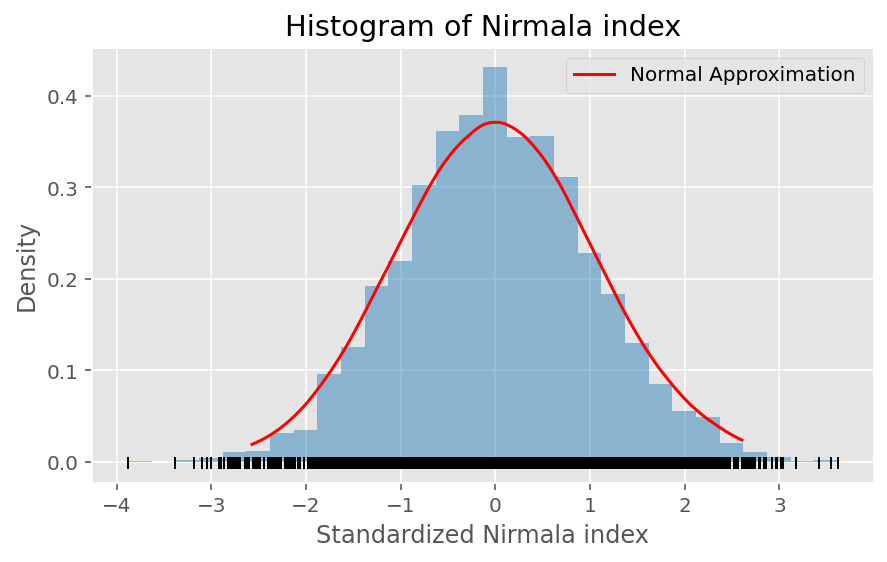}
    \caption{ Histogram of the standardized Nirmala index of $5,000$ independently
generated random spiro chains with $n = 10,000$; the thick red curve
is the estimated density of the sample.} 
    \label{f9}
\end{figure}

\begin{figure}[h!]

   \centering
    \includegraphics[width=0.5\textwidth]{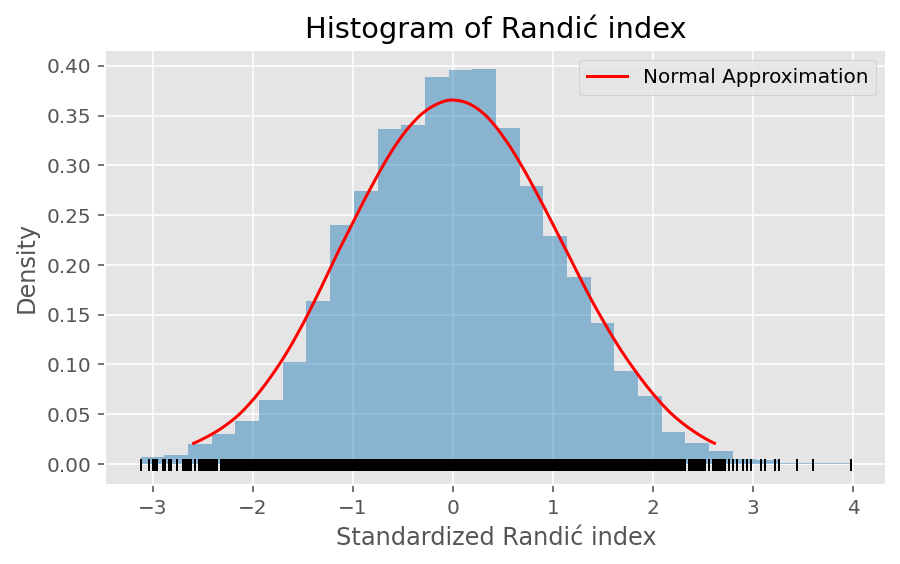}
    \caption{ Histogram of the standardized Randić index of $5,000$ independently
generated random spiro chains with $n = 10,000$; the thick red curve
is the estimated density of the sample.} 
    \label{f10}
\end{figure}

\newpage

\begin{figure}[h!]

   \centering
    \includegraphics[width=0.5\textwidth]{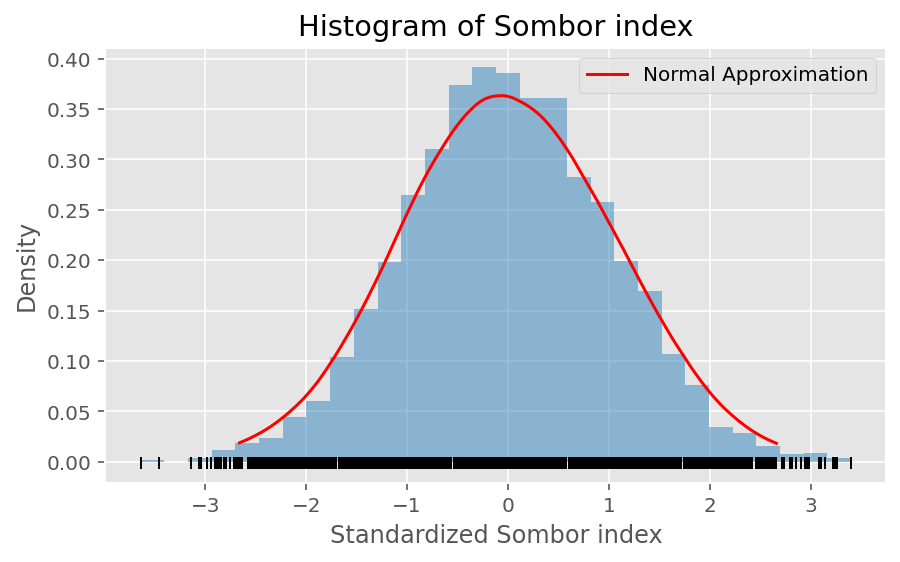}
    \caption{ Histogram of the standardized Sombor index of $5,000$ independently
generated random spiro chains with $n = 10,000$; the thick red curve
is the estimated density of the sample.} 
    \label{f11}
\end{figure}

\begin{figure}[h!]

   \centering
    \includegraphics[width=0.5\textwidth]{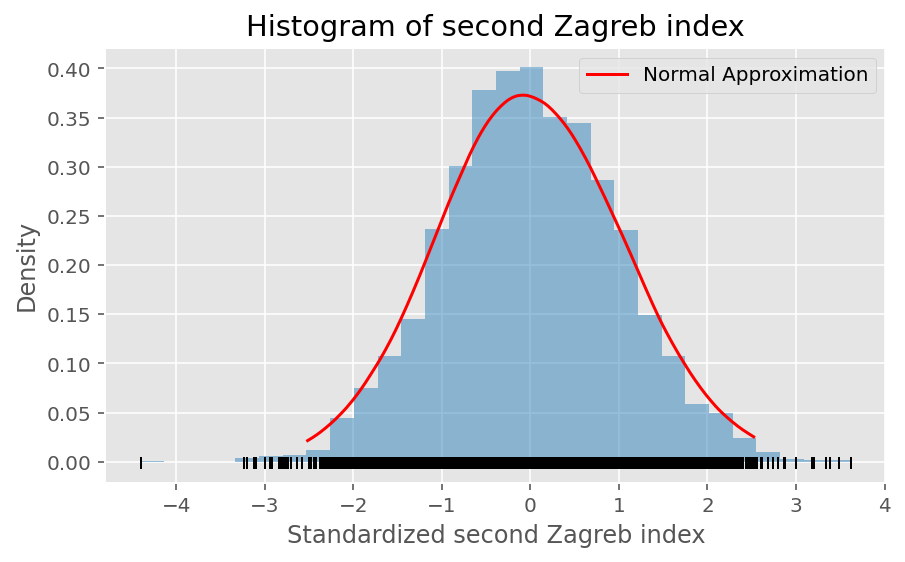}
    \caption{ Histogram of the standardized second Zagreb index of $5,000$ independently
generated random spiro chains with $n = 10,000$; the thick red curve
is the estimated density of the sample.} 
    \label{f12}
\end{figure}

\section{Concluding Remarks}
In this paper, we propose a martingale
approach to the study of topological indices in random spiro chains. The expected value, variance, exact distribution have been determined. Also, we formulate a martingale to characterize the asymptotic behavior of the topological indices. We show that the same analysis works here if we simply use a martingale central limit theorem instead of a classical central limit theorem. Moreover, we consider some particular topological indices, such as, Nirmala, Sombor, Randić and Zagreb index, in other words, we exploit the martingale approach.

\bibliographystyle{chicago}
\bibliography{Biblioteca}

\end{document}